\documentclass[10pt,a4paper,twoside]{article}

\makeatletter
\let\@fnsymbol\@arabic
\makeatother

\usepackage{amsmath,amsfonts,amssymb,latexsym,wasysym,mathrsfs}
\usepackage{amsthm}
\usepackage{verbatim}
\usepackage[english]{babel}
\usepackage{a4wide}
\usepackage{pstricks,graphicx,subfigure}
\usepackage{authblk}

\usepackage[footnotesize,bf,centerlast]{caption}
\usepackage[small,euler-digits,icomma,OT1,T1]{eulervm}
\usepackage[all,2cell]{xy} \UseAllTwocells \SilentMatrices
\usepackage{todonotes}
\usepackage[normalem]{ulem}
\usepackage{enumitem}

\theoremstyle{plain}
\newtheorem{theorem}{Theorem}[section]

\newtheorem{lemma}{Lemma}[section]

\theoremstyle{definition}

\theoremstyle{remark}

\newtheorem*{remark*}{Remark}

\title{Effects of local fields in a dissipative Curie-Weiss model: Bautin bifurcation and large self-sustained oscillations}

\author[1,2]{Francesca Collet}
\author[2,3]{Marco Formentin}

\affil[1]{\footnotesize Delft Institute of Applied Mathematics, Delft University of Technology, van Mourik Broekmanweg 6, 2628~XE~Delft (The Netherlands).}
\affil[2]{\footnotesize Dipartimento di Matematica ``Tullio Levi-Civita'', via Trieste 63, 35121 Padova (Italy).}
\affil[3]{\footnotesize Padova Neuroscience Center, via Giuseppe Orus 2, 35131 Padova (Italy) \newline \newline \vspace{.3cm} {\small \emph{E-mail addresses}:~$\{$francesca.collet, marco.formentin$\}$@unipd.it}}

\date{}

\begin{document}

\maketitle

\begin{abstract}
\noindent We modify the spin-flip dynamics of a Curie-Weiss model with dissipative interaction potential \cite{DaPFiRe13} by adding a site-dependent i.i.d. random magnetic field. 
The purpose is to analyze how the addition of the field affects the time-evolution of the observables  in the macroscopic limit. 
Our main result shows that a Bautin bifurcation point exists and that, whenever the field intensity is sufficiently strong and the temperature sufficiently low, a periodic orbit  emerges through a global bifurcation in the phase space, giving origin to a large-amplitude rhythmic behavior. 

\vspace{0.3cm}

\noindent {\bf Keywords:} Bautin bifurcation $\cdot$ collective noise-induced periodicity $\cdot$ disordered systems $\cdot$ mean-field interaction $\cdot$ non-equilibrium systems $\cdot$ random potential $\cdot$ saddle-node bifurcation of periodic orbits   \\ \\
\end{abstract}

\section{Introduction}

Large volume dynamics of noisy interacting units may display robust collective periodic behavior. Self-sustained oscillations are commonly encountered in ecology \cite{Tur03}, neuroscience \cite{ErTe10,LiG-ONeS-G04} and socioeconomics \cite{WeHa12}. 
From a modelling point of view, great attention has been given to mean-field interacting particle systems, due to their analytical tractability. In this context, the attempt of explaining rigorously possible origins of self-organized rhythms identified various essential aspects to enhance the emergence of  such coherent and structured dynamics.  Seminal works \cite{Sch85a,Sch85b} have highlighted the importance of the interplay between interaction and noise. In particular, the role of noise is twofold: on the one hand, noise can  lead to oscillatory states in systems whose deterministic counterparts do not display any periodic behavior (\emph{noise-induced periodicity}) \cite{Sch85a,Sch85b,ToHeFa12}; on the other, it can facilitate the transition from incoherence to  macroscopic pulsing (\emph{excitability by noise}) \cite{CoDaPFo15,LiG-ONeS-G04,LuPo_a,LuPo_b}. Moreover, rhythmic behaviors are intrinsically non-equilibrium phenomena, naturally in contrast with stochastic reversibility \cite{BeGiPa10,GiPo15}, and hence a reversibility-breaking mechanism needs to enter the microscopic design of the models. Quite a number of such mechanisms have been taken into account: addition of a \emph{driving force} or of a \emph{random intrinsic frequency} in phase rotator systems \cite{BoNeSp92,GiPaPePo12,ShKu86} ; addition of \emph{delay} in the information transmission and/or \emph{frustration} in the interaction network in multi-population discrete particle systems \cite{AlMi,CoFoTo16,DiLo17,FeFoNe09,Tou}; to name a fews.
A further mechanism  that has been lately proposed and investigated is \emph{dissipation}. Indeed, a class of irreversible models can be obtained as  perturbation of classical reversible dynamics by introducing a friction term damping the interaction potential \cite{AnTo18,CoDaPFo15,DaPFiRe13}. The simplest interacting particle system within this  family is the dissipative version of the Curie-Weiss model presented in \cite{DaPFiRe13}: the standard spin-flip dynamics are modified so that the interaction energy undergoes a dissipative and diffusive stochastic evolution. In the infinite volume limit, for sufficiently strong interaction and zero (or sufficiently small) noise, this system exhibits stable self-sustained oscillations \ emerging via a Hopf bifurcation.\\
In the present paper we modify the noiseless version of the dissipative Curie-Weiss model  introduced in \cite{DaPFiRe13} by adding some disorder: we embed the particle system in a site dependent, i.i.d., binary and symmetric, static random environment. Despite the simple structure of the random field we are considering, the addition of local fields makes the  phase diagram of the macroscopic evolution very rich and interesting. In particular, in the parameter space there exists a Bautin bifurcation point that allows to identify two half-planes for the field intensity corresponding to a small and a large-disorder regime, where the impact of the random environment is irrelevant and relevant respectively. In the small-disorder case, the system behaves as in absence of random field. As for the homogeneous model \cite[Thm.~3.1]{DaPFiRe13}, the emergence of self-sustained oscillations is due to the occurrence of a supercritical Hopf bifurcation.  On the contrary, whenever the disorder becomes sufficiently strong and the temperature is sufficiently low, a stable periodic orbit arises through a \emph{saddle-node bifurcation of limit cycles} rather than a Hopf bifurcation. From a technical viewpoint, we would like to stress that, exploiting the special structure of the vector field governing the infinite volume dynamics, we are able to extend globally the fold of limit cycles prescribed only locally by the Bautin bifurcation.\\


The paper is organized a follows. In Section~\ref{sect:model_and_results} we describe the model under consideration and we state our main results. All the proofs are postponed to Section~\ref{sect:proofs}.

\bigskip

\section{Description of the model and results}
\label{sect:model_and_results}

We consider a simplified version of the Curie-Weiss model with dissipation in \cite{DaPFiRe13} and we introduce inhomogeneity in the structure of the system via a site-dependent, static, random magnetic field (acting as a random environment).
Let $\sigma = \left( \sigma_j \right)_{j=1}^N \in \{-1,+1\}^N$ denote the spin configuration and let $\eta = \left( \eta_j \right)_{j=1}^N \in \{-1,+1\}^N $, a sequence of  i.i.d. random variables with distribution \mbox{$\mu = \frac{1}{2} (\delta_{-1} + \delta_{+1})$}, denote the disorder. Given a realization of the environment $\eta$, the stochastic process $\{\sigma(t)\}_{t \geq 0}$ is described as follows. For $\sigma \in \{-1,+1\}^N$, let us define $\sigma^i$ the configuration obtained from $\sigma$ by flipping the $i$-th spin. The spins will be assumed to evolve with one spin-flip dynamics: at any time $t$, the system may experience a transition $\sigma \longrightarrow \sigma^i$ at rate $1 -  \tanh [\sigma_i (\lambda_N + h \eta_i)]$, where $h \geq 0$ and $\{ \lambda_N (t) \}_{t \geq 0}$ is a stochastic process on $\mathbb{R}$, driven by the stochastic differential equation
\begin{equation}\label{eqn:evolution_of_lambda}
d \lambda_N(t) = - \lambda_N(t) dt  + \beta dm_N^{\sigma}(t) \,,
\end{equation}
with $\beta > 0$ and 
$m_N^{\sigma}(t) = \frac{1}{N} \sum_{j=1}^N \sigma_j(t)$. As far as the parameters: $\beta$ is the inverse temperature and $h$ is the intensity of the local fields. \\
The two terms in the argument of the hyperbolic tangent have different effects: the first one encodes the ferromagnetic coupling between spins, while the second pushes each spin to point the direction prescribed by the field associated with its own site. Observe that, in view of the evolution \eqref{eqn:evolution_of_lambda}, the interaction is damped between two consecutive spin flips.

From a formal viewpoint, for any given realization of $\eta$, we are considering  the Markov process $(\sigma(t), \lambda_N(t))$ on $\{-1,+1\}^N \times \mathbb{R}$ evolving with infinitesimal generator 
\begin{equation}\label{MicroSyst:InfGen}
L^{\eta}_N f (\sigma, \lambda) = \sum_{j=1}^N \left[ 1 -  \sigma_j \tanh ( \lambda + h \eta_j) \right] \left[ f \left( \sigma^j, \lambda  -  \tfrac{2\beta \sigma_j}{N}  \right) - f(\sigma, \lambda)\right] -\lambda \partial_{\lambda} f (\sigma,, \lambda) \,.
\end{equation}
In addition to the usual empirical magnetization, we define also the empirical averages
\[
m^{\sigma\eta}_{N}(t) = \frac{1}{N} \sum_{j=1}^N  \sigma_j(t) \, \eta_j \quad \mbox{ and } \quad \bar{\eta}_{N} = \frac{1}{N} \sum_{j=1}^N \eta_j \,.
\]
Let $E_N$ be the image of $\{-1,+1\}^{N} \times \mathbb{R}$ under the map ${\Phi^{\eta}:} (\sigma,\lambda_N) \mapsto (m_N^{\sigma},m_N^{\sigma \eta},\lambda_N)$. The microscopic dynamics on the configurations, corresponding to the generator \eqref{MicroSyst:InfGen}, induce a Markovian evolution on $E_N$ for the process $\{(m_N^{\sigma}(t), m_N^{\sigma\eta}(t), \lambda_N(t))\}_{t \geq 0}$, that in turn evolves with generator
\begin{multline}\label{InfGen:OrdPar}
\mathcal{L}_N f \left( m^{\sigma}, m^{\sigma\eta}, \lambda \right) = \sum_{j,k = \pm 1}  C_{N}(j,k)   \left[f\left(m^{\sigma} - \tfrac{2j}{N}, m^{\sigma\eta} -\tfrac{2jk}{N}, \lambda  -  \tfrac{2\beta j}{N} \right)- f \left( m^{\sigma}, m^{\sigma\eta}, \lambda \right)\right] \\
- \lambda \partial_{\lambda} f \left(m^{\sigma}, m^{\sigma\eta}, \lambda \right) \,,
\end{multline}
where $C_{N}(j,k) = \frac{N}{4} \left(1 + k \bar{\eta}_{N} + j m^{\sigma} + jk m^{\sigma\eta} \right) \left[ 1  -  j \tanh(\lambda + kh) \right]$, for all $j,k \in \{-1,+1\}$. Observe that the term $\frac{N}{4}(\dots)$ counts the number of pairs $(\sigma_i,\eta_i)$, $i\in\{1,\dots,N\}$, such that $\sigma_i=j$ and $\eta_i=k$.

The generator \eqref{InfGen:OrdPar} can be derived from \eqref{MicroSyst:InfGen} via the martingale problem and the property 
\[
L^{\eta}_N (f \circ \Phi^{\eta}) (\sigma, \lambda) = (\mathcal{L}_N f) \circ \Phi^{\eta} (\sigma, \lambda) = (\mathcal{L}_N f)(m^{\sigma}, m^{\sigma\eta}, \lambda).
\]
The process $\{(m_N^{\sigma}(t), m_N^{\sigma\eta}(t), \lambda_N(t))\}_{t \geq 0}$ is an \emph{order parameter}, in the sense that its dynamics completely describe the dynamics of the original system. We are going to characterize its limiting evolution.
We can derive the infinite volume dynamics for our model via weak convergence in $\mathcal{D}_{\mathbb{R}^3}(\mathbb{R}^+)$, the space of c{\`a}dl{\`a}g trajectories from $\mathbb{R}^+$ to $\mathbb{R}^3$.
\begin{theorem}[Law of large numbers]\label{thm:macro_evolution}
Suppose that $( m_N^{\sigma}(0), m_N^{\sigma\eta}(0), \lambda(0) )$ converges weakly to the constant $(m_0^{\sigma}, m_0^{\sigma\eta},\lambda_0)$. Then, $\mu$-almost surely, the stochastic process $\{(m_N^{\sigma}(t), m_N^{\sigma\eta}(t), \lambda_N(t))\}_{t \geq 0}$ converges weakly in $\mathcal{D}_{\mathbb{R}^3}(\mathbb{R}^+)$ to the unique solution of
%
%
\begin{equation}\label{MKV:equation}
\left\{
\begin{array}{lcl}
\dot{m}^{\sigma}_t &=& -2 m^{\sigma}_t  +  \tanh (\lambda_t + h)  +  \tanh (\lambda_t - h)\\
\dot{m}^{\sigma \, \eta}_t &=& -2 m^{\sigma \, \eta}_t  +  \tanh (\lambda_t + h)  -  \tanh (\lambda_t - h)\\
\dot{\lambda}_t &=& - \lambda_t  + \beta \left[  -  2 m^{\sigma}_t + \tanh (\lambda_t + h) + \tanh (\lambda_t - h) \right]. 
\end{array}
\right.
\end{equation}
\end{theorem} 

Next we want to characterize the phase diagram for \eqref{MKV:equation}. Observe that the first and third equation in \eqref{MKV:equation} form an independent subsystem. In other words, we are reduced to analyze the attractors on $[-1,+1] \times \mathbb{R}$ for the planar system
\begin{equation}\label{MKV:reduced}
(\dot{m}^{\sigma}_t,\dot{\lambda}_t) = V (m^{\sigma}_t, \lambda_t) \,,
\end{equation} 
with vector field given by
\[
V(x,y) = \big(-2x  +  \tanh(y+h)  +  \tanh(y-h), \, -y + \beta \big[  -  2 x + \tanh (y+h) + \tanh (y-h) \big] \big).
\]
%
We remark that the only fixed point for \eqref{MKV:reduced} is the origin. Moreover, it is easy to see that $(0,0)$ is linearly stable whenever $\beta < \frac{3}{2} \cosh^2 (h)$; whereas, the local stability is lost for $\beta > \frac{3}{2} \cosh^2 (h)$. Much more than a local analysis can be obtained for \eqref{MKV:reduced}. The global situation is described in the next theorem and then qualitatively illustrated in Figure~\ref{Fig:PD}.

\begin{theorem}[Phase Diagram]\label{Thm_PD} 
Consider the dynamical system \eqref{MKV:reduced} and, for every $h \geq 0$, set $\beta_c(h) := \frac{3}{2} \cosh^2 (h)$ and $h_{\mathrm{tc}} := \frac{1}{2} \ln (2 + \sqrt{3})$. We have the following:
\begin{enumerate}[label=(\alph*),ref=(\alph*)]
\item \label{Thm_PD:local_bifurcation_case}
Suppose $h \leq h_{\mathrm{tc}}$. Then,
\begin{enumerate}
\item 
for $\beta \leq \beta_c(h)$ the origin is a global attractor.
\item 
for $\beta > \beta_c(h)$ the system admits a unique limit cycle attracting all the trajectories except for the fixed point.
\end{enumerate}
\item \label{Thm_PD:global_bifurcation_case}
Suppose $h > h_{\mathrm{tc}}$. Then, there exists $0 < \beta_{\star}(h) \leq \beta_c(h)$ such that
\begin{enumerate}
\item 
for $\beta < \beta_{\star}(h)$ the origin is a global attractor.
\item 
for $\beta_{\star}(h) \leq \beta < \beta_c(h)$ the origin is locally stable and coexists with a stable periodic orbit.
\item
for $\beta \geq \beta_c(h)$ the system admits a unique limit cycle attracting all the trajectories except for the fixed point.
\end{enumerate}
\end{enumerate}
In particular, the tricritical point $( h_{\mathrm{tc}}, \beta_{\mathrm{tc}}) := ( h_{\mathrm{tc}}, \beta_{\mathrm{c}}(h_{\mathrm{tc}})) = ( \frac{1}{2} \ln (2+\sqrt{3}), \frac{9}{4} )$ is a Bautin bifurcation point.
\end{theorem}

\begin{figure}[h]
\centering
\includegraphics[scale=1]{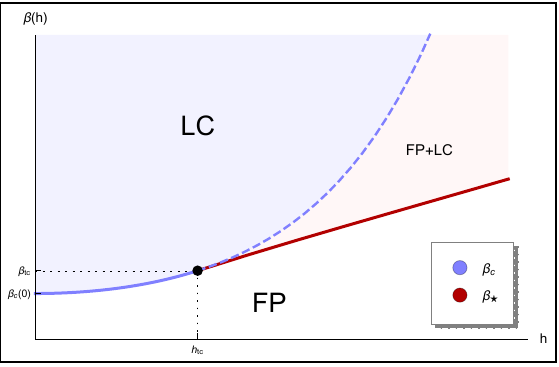}
\caption{
Illustration of the phase portrait for the dynamical system \eqref{MKV:reduced}. Each colored region represents a phase with attractor(s) indicated by the label: FP = fixed point; LC = limit cycle; FP+LC = coexistence of fixed point and limit cycle. The blue separation curve $\beta_{\mathrm{c}} = \frac{3}{2} \cosh^2(h)$ is the Hopf bifurcation curve; in particular, it is solid when the bifurcation is supercritical and dashed otherwise. The red  curve has been obtained numerically and represents the saddle-node bifurcation of periodic orbits.
The tricritical point $(h_{\mathrm{tc}},\beta_{\mathrm{tc}})$  is a Bautin bifurcation point.
}
\label{Fig:PD}
\end{figure}

\begin{remark*}
Finding, for any  $h > h_{\mathrm{tc}}$, the exact threshold value $\beta_{\star}(h)$ is hard  to achieve analytically.
On the contrary, it is easy to get a positive lower bound for such a transition point. See Appendix~\ref{App:study_function_g} for further details.
\end{remark*} 


We give a few explanations concerning the content of Theorem~\ref{Thm_PD}. All the technicalities can be found in Section~\ref{subsect:proof_thm_PD}. If $h \leq 
 h_{\mathrm{tc}}$ a periodic orbit appears through a local change in the stability of the fixed point. At $\beta = \beta_{\mathrm{c}}(h)$ a \emph{supercritical Hopf} bifurcation occurs. When $h > h_{\mathrm{tc}}$ the dependence of the attractors on the parameter $\beta$ is quite nontrivial, due to the combination of a \emph{subcritical Hopf} bifurcation and a \emph{saddle-node} bifurcation of limit cycles. Roughly speaking, there are three possible regimes for system \eqref{MKV:reduced}.
\begin{itemize}
\item
Fixed point phase. For $\beta < \beta_\star(h)$ the only stable attractor is the origin. 
\item
Coexistence phase. At $\beta = \beta_\star(h)$ a semistable cycle surrounding the origin is formed. By increasing the parameter $\beta$ from $\beta_\star(h)$, this cycle splits into two limit cycles, the outer being stable and the inner unstable. In this phase $(0, 0)$ is linearly stable. Therefore, the locally stable equilibrium coexists with a stable periodic orbit. 
\item
Periodic orbit phase. For $\beta = \beta_{\mathrm{c}}(h)$ the (subcritical) Hopf bifurcation occurs: the inner unstable limit cycle collapses at $(0,0)$ and disappears. At the same time the fixed point loses its stability and, thus, the external stable limit cycle is the only stable attractor left for $\beta \geq \beta_{\mathrm{c}}(h)$.
\end{itemize}

Notice that even if the two scenarios depicted in Theorem~\ref{Thm_PD} may look similar, as they both describe a transition of the type ``fixed point to limit cycle", this is not the case. They are in fact qualitatively very different.  If $h \leq h_{\mathrm{tc}}$ a \emph{small-amplitude periodic orbit} bifurcates from the origin at the critical point and then it grows gradually when increasing the parameter $\beta$. On the contrary, if $h >  h_{\mathrm{tc}}$ the stable limit cycle arises through a \emph{global bifurcation} and therefore, when they originate, the \emph{oscillations are already large}. As a consequence the trajectories are abruptly pushed far from the equilibrium point when the latter becomes unstable.\\

Recall we are assuming that the disorder is modeled by random variables with a {\em symmetric alphabet} $\{-1,+1\}$ and a {\em zero-mean distribution} $\mu=\frac{1}{2}\left(\delta_{-1}+\delta_{+1}\right)$. These two properties are key ingredients to achieve the global results in the phase diagram described in Theorem~\ref{Thm_PD}.  On the one hand, the fact that the field takes two symmetric values allows for algebraic cancellations (which otherwise do not take place) in the expansion of the generator that  lead to a decoupling of the macroscopic evolutions of the observables, reducing in fact a 3-dimensional system of ODEs to a 2-dimensional one. On the other hand, the symmetry of the distribution seems to be essential to profit from standard Li{\'e}nard's theorem \cite[Thm.~1, Sect.~3.8]{Per01} and easily control existence, uniqueness and stability of the limit cycle above the critical curve $\beta=\beta_{\mathrm{c}}(h)$ and far from the Bautin bifurcation point $(h_{\mathrm{tc}}, \beta_{\mathrm{tc}})$. However we believe that for the phase diagram in Figure~\ref{Fig:PD} to hold, the zero-mean distribution assumption may be dropped. It can be checked that the  Hopf bifurcation curve and the Bautin bifurcation point exist even when taking $\eta_i \in \{-1,+1\}$ distributed, more generally, according to $\mu=p\delta_{-1}+(1-p)\delta_{+1}$, for some $p\in[0,1]$. Actually the critical curve $\beta=\beta_c(h)$ and the tri-critical point $(h_{\mathrm{tc}}, \beta_{\mathrm{tc}})$ are the same for all the values of $p$. The local analysis can be carried out exactly as done in Section~\ref{Thm_PD} for $p=\frac{1}{2}$. Whenever $p \neq \frac{1}{2}$, the lack of symmetry makes very hard providing instead a rigorous global analysis. Nevertheless, numerical simulations suggest that also in this case, for any value of $h>h_{tc}$, a stable periodic orbit arises through a saddle-node bifurcation of limit cycles, as shown in Figure~\ref{Fig:PD_gen}. \\

To conclude, we can see that the addition of local fields has a significant impact as the nature of the bifurcation may be drastically changed for large enough field intensity and sudden self-sustained large-amplitude oscillations may be induced in the system.  In the $h=0$ case the emergence of the limit cycle always and only occurs through a supercritical Hopf, hence local, bifurcation \cite{DaPFiRe13}.

\begin{figure}[h]
\centering
\includegraphics[width=.6\textwidth]{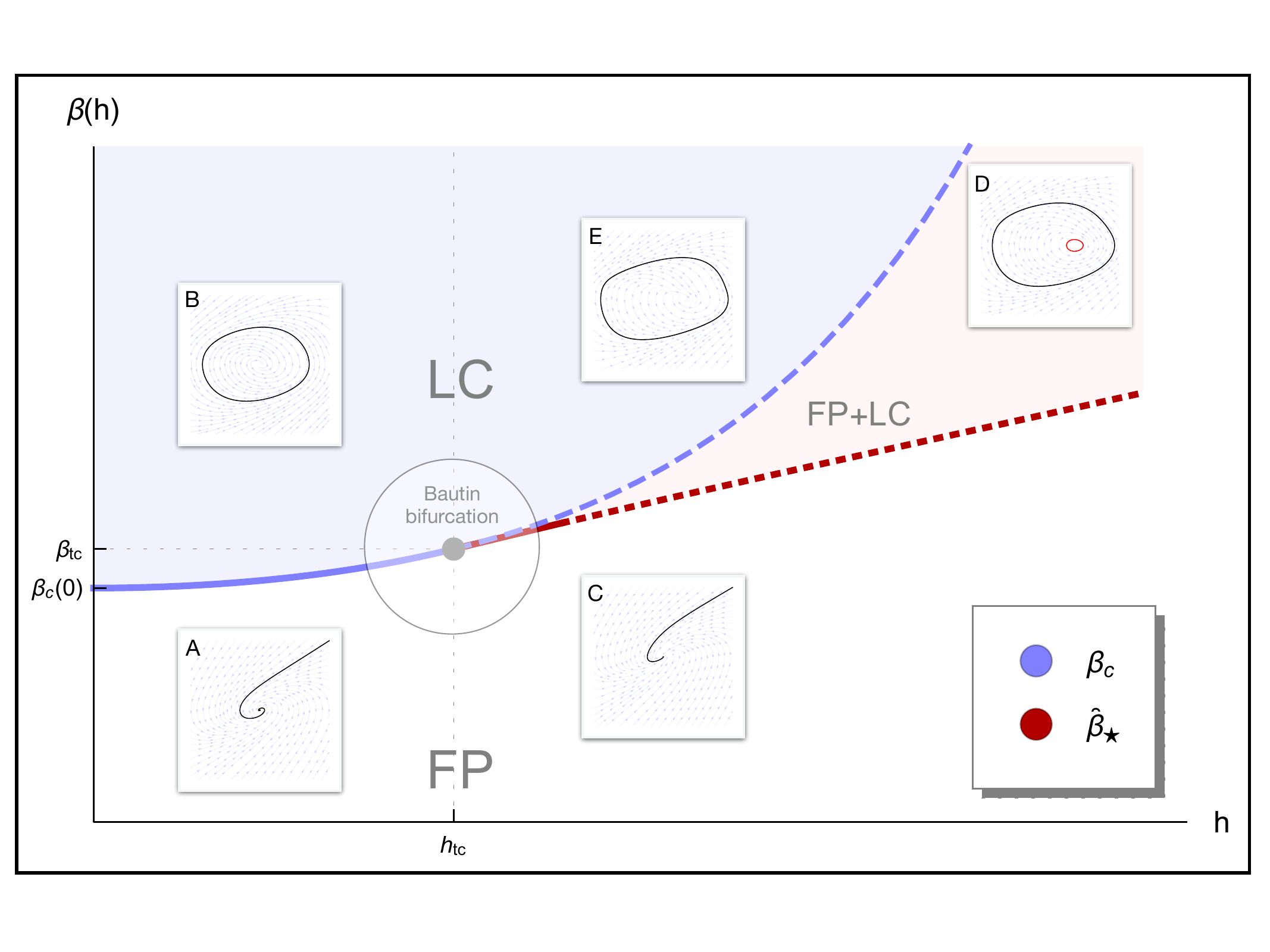}
\caption{Phase diagram for the macroscopic limit of the particle system \eqref{MicroSyst:InfGen} when choosing field distribution $\mu = p \delta_{-1} + (1-p) \delta_{+1}$ with $p=\frac{1}{3}$. Each colored region represents a phase with attractor(s) indicated by the label: FP = fixed point; LC = limit cycle; FP+LC = coexistence of fixed point and limit cycle. The blue separation curve $\beta_{\mathrm{c}} = \frac{3}{2} \cosh^2(h)$ is the Hopf bifurcation curve; in particular, it is solid when the bifurcation is supercritical and dashed otherwise. The dotted red line $\hat{\beta}_{\star}$ is qualitative, $p$-dependent and represents the saddle-node bifurcation of periodic orbits. The tricritical point $(h_{\mathrm{tc}},\beta_{\mathrm{tc}})$ is a Bautin bifurcation point. 
For each phase the vector field and a representative trajectory/attractor are displayed (in panel ~D also the unstable limit cycle is shown and it is colored red). They have been obtained numerically by setting the parameters as follows. Panel~A: $h=0.2$, $\beta=1$. Panel~B: $h=0.2$, $\beta=2$. Panel~C: $h=1$, $\beta=1$.   Panel~D: $h=1$, $\beta=4.5$. Panel~E: $h=1$, $\beta=3.5$.}
\label{Fig:PD_gen}
\end{figure}

\section{Proofs}\label{sect:proofs}

\subsection{Proof of Theorem~\ref{thm:macro_evolution}}
\label{subsect:proof_thm_PD}

To prove weak convergence in path-space we combine a compact containment condition with the convergence of generators.

\paragraph{Compact containment condition.} Notice that the processes $\{m_N^{\sigma}(t)\}_{t \geq 0}$ and $\{m_N^{\sigma \eta}(t)\}_{t \geq 0}$ are confined in $[-1,+1] \subset \mathbb{R}$ for each $N \in \mathbb{N}$. We are thus left with showing that the sequence of processes $\{\lambda_N(t)\}_{t \geq 0}$ is contained in a compact; we do it by proving compact containment for the process $\{\beta m_N^{\sigma}(t) - \lambda_N(t)\}_{t \geq 0}$.

For every $N \in \mathbb{N}$, let us define the stopping time $\tau_N^M := \inf \{ t \geq 0: \vert \beta m_N^{\sigma}(t) - \lambda_N(t) \vert \geq M\}$. We study the asymptotic behavior of the sequence $\{ \tau_N^M \}_{N \geq 1}$.

\begin{lemma}\label{lmm:compact_containment}
For any $T>0$ and $\varepsilon > 0$, there exists $M_{\varepsilon}>0$ such that $\sup_{N \geq 1} P(\tau_N^{M_{\varepsilon}} \leq T) \leq \varepsilon$.
\end{lemma}

\begin{proof}
Let $M$ be an arbitrary strictly positive constant. Observe that
\[
P \left( \tau_N^M \leq T \right) = P \left( \sup_{0 \leq t \leq T \wedge \tau_N^M} \left\vert \beta m_N^{\sigma}(t) - \lambda_N(t) \right\vert \geq M \right).
\]
We want to show that the probability in the right-hand side of the previous display can be made arbitrarily small.

Consider the function $U: [-1,+1] \times \mathbb{R} \to \mathbb{R}$, given by $U(m^\sigma,\lambda) = \frac{1}{2}(\beta m^\sigma - \lambda)^2$. Since $U(m^\sigma \pm \frac{2}{N}, \lambda \pm \frac{2\beta}{N}) = U(m^\sigma,\lambda)$, the evolution of the process $\left\{ U \left( m_N^\sigma(t), \lambda_N(t) \right) \right\}_{t \geq 0}$ is deterministic and driven by the generator
\[
\mathcal{L}_N U(m^\sigma,\lambda) = - \lambda \partial_\lambda U(m^\sigma,\lambda) = - \lambda^2 + \beta m^\sigma \lambda,
\]
cf. equation \eqref{InfGen:OrdPar}. Notice that $\mathcal{L}_N U(m^\sigma,\lambda) \leq \frac{\beta^2}{4}$. Therefore, we have
\begin{align*}
U \left(m_N^\sigma(t), \lambda_N(t)\right) &= U \left(m_N^\sigma(0), \lambda_N(0)\right) + \int_0^t \mathcal{L}_N U \left( m_N^\sigma(s), \lambda_N(s) \right) ds \\[0.2cm]
&\leq U \left(m_N^\sigma(0), \lambda_N(0)\right) + \frac{\beta^2 t}{4},
\end{align*}
leading to
\begin{align*}
P \left( \sup_{0 \leq t \leq T \wedge \tau_N^M} \left\vert \beta m_N^{\sigma}(t) - \lambda_N(t) \right\vert \geq M \right) &= P \left( \sup_{0 \leq t \leq T \wedge \tau_N^M} U \left(m_N^{\sigma}(t),\lambda_N(t) \right) \geq \frac{M^2}{2} \right)\\[0.2cm]
&\leq P \left( U \left(m_N^\sigma(0), \lambda_N(0)\right) \geq \frac{M^2}{2} - \frac{\beta^2 T}{4} \right).
\end{align*}
The convergence in law of the initial condition implies $P(U \left(m_N^\sigma(0), \lambda_N(0)\right) \geq c(\varepsilon)) \leq \varepsilon$ for a sufficiently large $c(\varepsilon) >0$ and all $N \in \mathbb{N}$. As a consequence, to conclude it suffices to choose the constant $M = M_{\varepsilon}$ so that $\frac{M^2}{2} - \frac{\beta^2 T}{4} \geq c(\varepsilon)$.
\end{proof}

\paragraph{Convergence of the sequence of generators.} The infinitesimal generator of the process $\{(m_N^{\sigma}(t), m_N^{\sigma\eta}(t),\lambda_N(t))\}_{t \geq 0}$ is given in \eqref{InfGen:OrdPar}. We want to characterize the limit of the sequence $\{\mathcal{L}_N f\}_{N \geq 1}$ for $f \in C_c^2([-1,+1]^2 \times \mathbb{R})$, the set of two times continuously differentiable functions that are constant outside a compact set in the interior of $[-1,+1]^2 \times \mathbb{R}$. We first Taylor expand $f$ up to first order. For all $j, k \in \{-1,+1\}$, we get
\begin{multline*}
f \left( m^{\sigma} - \tfrac{2j}{N}, m^{\sigma \eta} - \tfrac{2jk}{N}, \lambda  -  \tfrac{2\beta j}{N} \right) - f \left(m^{\sigma}, m^{\sigma\eta}, \lambda \right) \\
=  - \tfrac{2j}{N} \, \partial_{m^{\sigma}} f \left(m^{\sigma}, m^{\sigma\eta}, \lambda \right)  - \tfrac{2jk}{N} \, \partial_{m^{\sigma\eta}} f \left(m^{\sigma}, m^{\sigma\eta}, \lambda \right) \\
 -  \tfrac{2\beta j}{N} \, \partial_\lambda  f \left(m^{\sigma}, m^{\sigma\eta}, \lambda \right) + O \left(\tfrac{1}{N}\right)
\end{multline*}
and then, by combining the terms with $\partial_{m^{\sigma}}$, with $\partial_{m^{\sigma\eta}}$ and the terms with $\partial_\lambda$, we get
\begin{align*}
\mathcal{L}_N f \left( m^{\sigma}, m^{\sigma\eta}, \lambda \right) &= \big\{ -2 m^{\sigma}  +  \left( 1 + \bar{\eta}_N \right) \tanh (\lambda + h)  +  \left( 1 - \bar{\eta}_N \right) \tanh (\lambda - h) \big\} \, \partial_{m^\sigma} f(-) \\
& + \big\{ -2 m^{\sigma\eta}  +  \left( 1 + \bar{\eta}_N \right) \tanh (\lambda + h)  -  \left( 1 - \bar{\eta}_N \right) \tanh (\lambda - h) \big\} \, \partial_{m^{\sigma \eta}} f(-) \\
& + \big\{ -\lambda+  \beta \big[  -  2 m^{\sigma} + \left( 1 + \bar{\eta}_N \right) \tanh (\lambda + h) +  \left( 1 - \bar{\eta}_N \right) \tanh (\lambda - h) \big] \big\} \, \partial_{\lambda} f(-) + O(1). 
\end{align*}
Observe that, in the limit as $N \to \infty$, the empirical average $\bar{\eta}_N$ converges to zero $\mu$-almost surely by the law of large numbers. Let $\mathcal{L}$ be the linear generator
\begin{align*}
\mathcal{L} f \left( m^{\sigma}, m^{\sigma\eta}, \lambda \right) &= \big\{ -2 m^{\sigma}  +  \tanh (\lambda + h)  +  \tanh (\lambda - h) \big\} \, \partial_{m^\sigma} f(-) \\
& + \big\{ -2 m^{\sigma\eta}  +  \tanh (\lambda + h)  -  \tanh (\lambda - h) \big\} \, \partial_{m^{\sigma \eta}} f(-) \\
& + \big\{ -\lambda +  \beta \big[  -  2 m^{\sigma} + \tanh (\lambda + h) + \tanh (\lambda - h) \big] \big\} \, \partial_{\lambda} f(-). 
\end{align*}
Since, for every $f \in C_c^2([-1,+1]^2 \times \mathbb{R})$ and any compact $K \subset \mathbb{R}^3$, we have
\[
\lim_{N \to \infty} \, \sup_{(m^{\sigma}, m^{\sigma\eta}, \lambda) \in K \cap E_N} \, \vert \mathcal{L}_N f (m^{\sigma}, m^{\sigma\eta},\lambda) - \mathcal{L} f (m^{\sigma}, m^{\sigma\eta},\lambda) \vert = 0,
\]
we obtain the convergence of $\mathcal{L}_N$ to $\mathcal{L}$, as $N$ tends to infinity.\\

To derive the weak convergence result we apply \cite[Cor.~4.8.16]{EtKu86}. We check the assumptions of the corollary are satisfied:
\begin{itemize}
\item
By Lemma~\ref{lmm:compact_containment} the sequence of processes $\{(m_N^{\sigma}(t), m_N^{\sigma\eta}(t), \lambda_N(t))\}_{t \geq 0}$ satisfies the compact containment condition.
\item
The set $C_c^2([-1,+1]^2\times \mathbb{R})$ is an algebra that separates points.
\item
The martingale problem for the operator $(\mathcal{L},C_c^2([-1,+1]^2 \times \mathbb{R}))$ admits a unique solution by \cite[Thm.~8.2.6]{EtKu86}.
\end{itemize}
The conclusion then follows.\\


\subsection{Proof of Theorem~\ref{Thm_PD}}

\emph{Local analysis.} System \eqref{MKV:reduced} admits only the fixed point $(0,0)$ in the phase plane $(m^{\sigma},\lambda)$. The linearization around this point gives
\begin{equation}\label{MKV:linearized}
\begin{pmatrix}
\dot{m}^{\sigma}\\
\dot{\lambda}
\end{pmatrix}
=
\begin{pmatrix}
-2 &    \frac{2}{\cosh^2(h)} \\
 -  2\beta & \frac{2\beta}{\cosh^2(h)} - 1
\end{pmatrix}
\begin{pmatrix}
m^{\sigma} \\
\lambda
\end{pmatrix}
\end{equation}
ant the eigenvalues of the system are
\[
k_{\pm} = \tfrac{\beta}{\cosh^2(h)} - \tfrac{3}{2} \pm \sqrt{\left( \tfrac{\beta}{\cosh^2(h)} - \tfrac{3}{2} \right)^2 - 2} \,.
\]
These eigenvalues have both negative real part for $\beta < \frac{3}{2} \cosh^2(h)$ and both positive real part for $\beta > \frac{3}{2} \cosh^2(h)$. As a consequence, if $\beta < \frac{3}{2}  \cosh^2(h)$, the origin is linearly stable; whereas, for $\beta > \frac{3}{2}  \cosh^2(h)$, it loses local stability. At the critical point $\beta = \frac{3}{2} \cosh^2(h)$ a Hopf bifurcation occurs. \\

\emph{Lyapunov number.} Set $\beta = \beta_c(h) = \frac{3}{2} \cosh^2(h)$. To understand whether the Hopf bifurcation is sub- or supercritical we compute the first Lyapunov number (or Lyapunov coefficient) associated with the origin. 
For a system cast in normal form at the bifurcation an explicit formula for such a number is given, see \cite[Sect.~3.4]{GuHo83}.  \\
The dynamical system \eqref{MKV:reduced} takes its normal form with respect to the new variables $x = \sqrt{2} (\lambda  -  \beta_c m^{\sigma})$  and $\lambda = \lambda$. The transformation yields
\begin{equation}\label{MKV:Lyapunov}
\left\{
\begin{array}{l}
\dot{x}_t = -\sqrt{2} \, \lambda_t \\
\dot{\lambda}_t = \sqrt{2} \, x_t - g_{\beta_c,h} (\lambda_t),
\end{array}
\right.
\end{equation}
with $g_{\beta_c,h}(\lambda) = 3\lambda - \beta_c [\tanh(\lambda+h)+ \tanh(\lambda-h)]$. Given \eqref{MKV:Lyapunov}, we compute the first Lyapunov number $\ell_1$ by means of formula (3.4.11) in \cite[Sect.~3.4]{GuHo83}. We readily obtain
\[
\ell_1 = \frac{3 \left[\cosh(2h)-2\right]}{8 \cosh^2(h)}.
\]
By \cite[Thm.~3.4.2, Sect.~3.4]{GuHo83}, we get that the Hopf bifurcation is subcritical whenever $\ell_1 > 0$ and supercritical otherwise. This corresponds to having subcriticality for $h < \frac{1}{2} \ln (2+\sqrt{3})$ and supercriticality for $h > \frac{1}{2} \ln (2+\sqrt{3})$. At the tricritical point $(h_{\mathrm{tc}}, \beta_{\mathrm{tc}}) = (\frac{1}{2} \ln (2 + \sqrt{3}), \frac{9}{4})$ we obtain $\ell_1 =0$. Then we determine the second Lyapunov number $\ell_2$ and repeat the same reasoning as above. Since $\ell_2 = - \frac{1}{360} < 0$, at the tricritical point the Hopf bifurcation is supercritical. We will sketch the steps to get $\ell_2$ in Appendix~\ref{App:second_Lyapunov_number}.\\
The above considerations also show that the point $(h_{\mathrm{tc}},\beta_{\mathrm{tc}})$ can be classified as a Bautin bifurcation point. It means that the scenario described in the statement of Theorem~\ref{Thm_PD} holds true locally, that is in a small neighborhood of the tricritical point, see \cite[Sect.~8.3]{Kuz98}. In the sequel we prove that the result is valid globally.\\


\emph{Global analysis.} By performing  the change of variables $y=2(\lambda  -  \beta m^{\sigma})$ and $\lambda=\lambda$, we can transform system \eqref{MKV:reduced} into the Li\'enard system 
\begin{equation}\label{MKV:Lienard}
\left\{
\begin{array}{l}
\dot{y}_t = -2  \lambda_t \\
\dot{\lambda}_t = y_t - g_{\beta, h} (\lambda_t),
\end{array}
\right.
\end{equation}
with $g_{\beta, h}(\lambda) = 3 \lambda - \beta \left[ \tanh(\lambda+h) + \tanh (\lambda-h)\right]$. Having at hand a Li{\'e}nard system is very convenient as many global properties of possible limit cycles are available and can be derived by studying the zeroes of the function $g_{\beta, h}$, see \cite[Sect.~3.8]{Per01}. For the sake of readability, we collect the interesting properties of function $g_{\beta,h}$ in Appendix~\ref{App:study_function_g}. We proceed with the analysis of the attractors. \\

\begin{enumerate}
\item
Case $h \leq \frac{1}{2} \ln (2 + \sqrt{3})$. In this small-disorder regime the system behaves as in absence of disorder. Exploiting properties \ref{fact:1} and \ref{fact:4} given in Appendix~\ref{App:study_function_g}, the phase diagram can be proven analogously to the case of the homogeneous model discussed in~\cite{DaPFiRe13}. We refer to the proof of Theorem~3.1 therein for details. \\
\item
Case $h > \frac{1}{2} \ln (2 +\sqrt{3})$. We show that in the current regime a saddle-node bifurcation of cycles occurs.

\begin{itemize}
\item[(iii)] 
$\beta \geq \beta_c(h)$. In this case the function $g_{\beta, h}$ is odd and has exactly one positive zero at $\lambda = \lambda^*$ (see properties \ref{fact:3} and \ref{fact:4} in Appendix~\ref{App:study_function_g}). Moreover, $g_{\beta,h}'(0) < 0$ and $g_{\beta,h}$ is monotonically increasing to infinity for $\lambda > \lambda^*$. Therefore, standard Li{\'e}nard's theorem guarantees existence and uniqueness of a stable periodic orbit. See \cite[Thm.~1, Sect.~3.8]{Per01}.\\ 
\item[(ii)]
The crucial ingredient for proving the next two statements is the particular structure of the vector field generated by \eqref{MKV:Lienard}. It defines a semicomplete one-parameter family of negatively \emph{rotated vector fields} (with respect to $\beta$, for fixed $h$), see \cite[Def.~1, Sect.~4.6]{Per01}. For dynamical systems depending on a parameter in this peculiar way, many results concerning bifurcations, stability and global behavior of limit and separatrix cycles are known \cite[Chap.~4]{Per01}. In particular, we will base our analysis on the following properties: 
\begin{enumerate}[label={\footnotesize (P\arabic*)},ref={\footnotesize (P\arabic*)}]
\item\label{property:1}
limit cycles expand/contract monotonically as the parameter $\beta$ varies in a fixed sense;
\item\label{property:2}
a limit cycle terminates either at a critical point or at a separatrix of~\eqref{MKV:Lienard};
\item\label{property:3}
cycles of distinct fields do not intersect.
\end{enumerate}
Properties \ref{property:1} and \ref{property:2} allow to explain the rise of a separatrix cycle whose breakdown is responsible for a saddle-node bifurcation of limit cycles at $\beta=\beta_{\star}(h)$. We have already proved that for $\beta \geq \beta_{\mathrm{c}}(h)$ a stable periodic orbit exists.  While decreasing $\beta$ from $\beta_{\mathrm{c}}(h)$ this orbit shrinks and, at the same time, the limit cycle appeared at the Hopf bifurcation point expands, until they collide producing the semistable cycle at $\beta=\beta_{\star}(h)$. Looking the same process forwardly, we see what is happening in this phase. When the separatrix splits increasing $\beta$ from $\beta_{\star}(h)$, it generates two limit cycles surrounding $(0,0)$. The inner periodic orbit is unstable (due to the subcritical Hopf bifurcation at $\beta=\beta_{\mathrm{c}}(h)$) and represents the boundary of the basin of attraction of the origin. Moreover, the outer limit cycle inherits the stability of the exterior of semistable cycle and so it is stable. See \cite[Thm.~2 and Fig.~1, Sect.~4.6]{Per01} for more details.\\
\item[(i)]  Consider the Lyapunov function $U(y, \lambda) := \frac{y^2}{4} + \frac{\lambda^2}{2}$. Observe that its total derivative $\dot{U}(y,\lambda) = - \lambda g_{\beta, h}(\lambda)$ is negative for every $y \in \mathbb{R}$ and for every $\lambda \geq \frac{2\beta}{3}$. Thus, there exists a stable domain for the flux of \eqref{MKV:Lienard} and, in particular, the trajectories can not escape to infinity as $t \to +\infty$.  To conclude it is sufficient to prove that in this regime the dynamical system \eqref{MKV:Lienard} does not admit a limit cycle. Indeed, the non-existence of periodic orbits together with the existence of a stable domain for the flux guarantee that every trajectory must converge to a fixed point as $t \to +\infty$.  \\
Therefore, it remains to show that no limit cycle exists for $\beta < \beta_{\star}(h)$. From properties \ref{property:1} and \ref{property:2} it follows that, as $\beta$ increases from $\beta_{\star}(h)$ to infinity, the external stable limit cycle expands and its motion covers the entire region outside the separatrix. Similarly, the inner unstable cycle contracts from it and terminates at the origin. As a consequence, for $\beta \geq \beta_{\star}(h)$ the whole phase space is covered by expanding or contracting periodic orbits. From property \ref{property:3} we can deduce that  no periodic trajectory may exist  for $\beta < \beta_{\star}(h)$, as such an orbit would cross some of the cycles present for $\beta \geq \beta_{\star}(h)$.
\end{itemize}
At $\beta=0$ the differential system \eqref{MKV:Lienard} has solution $y_t = c_1 e^{-t} + c_2 e^{-2t}$, $\lambda_t = 2 c_1 e^{-t} + c_2 e^{-2t}$ ($c_1, c_2 \in \mathbb{R}$), excluding the possible existence of periodic solutions. Therefore $\beta_{\star}(h) > 0$ for all $h > h_{\mathrm{tc}}$. This concludes the proof.
\end{enumerate}

\appendix 
\renewcommand{\thefigure}{\Alph{section}.\arabic{figure}}

\section{Derivation of the second Lyapunov number}
\label{App:second_Lyapunov_number}

To determine the second Lyapunov number $\ell_2$ associated with the origin we follow the center manifold approach in \cite[pp. 175--181]{Kuz98}. In particular, we rely on the review of the method done in \cite[Sect.~3]{SoMedeCB07}, which is stated very clearly and in detail. We keep the same notation as therein and we report here only the relevant quantities and the main steps.\\

Consider the dynamical system \eqref{MKV:Lyapunov} and set $h = h_{\mathrm{tc}} = \frac{1}{2} \ln (2+\sqrt{3})$ and $\beta_c = \beta_c (h_{\mathrm{tc}}) = \beta_{\mathrm{tc}} = \frac{9}{4}$. We Taylor expand the second equation of \eqref{MKV:Lyapunov} around $\lambda=0$ up to the fifth order. It yields
\begin{equation}\label{MKV:Taylor_expansion}
\left\{
\begin{array}{l}
\dot{x}_t = -\sqrt{2} \, \lambda_t \\
\dot{\lambda}_t = \sqrt{2} \, x_t - \frac{4}{15} \lambda_t^5 + o\left(\lambda_t^6\right).
\end{array}
\right.
\end{equation}
Referring to \cite[eqns. (13--17)]{SoMedeCB07}, from \eqref{MKV:Taylor_expansion} we get
\[
A =
\begin{footnotesize}\begin{pmatrix}
0 & -\sqrt{2} \\
\sqrt{2} & 0
\end{pmatrix}\end{footnotesize}, 
\quad
B = C = D =
\begin{footnotesize}\begin{pmatrix}
0\\
0
\end{pmatrix}\end{footnotesize}
\quad \text{and} \quad
E \left( 
\begin{footnotesize}\begin{pmatrix}
v_1\\
w_1
\end{pmatrix}\end{footnotesize}, 
\dots, 
\begin{footnotesize}\begin{pmatrix}
v_5\\
w_5
\end{pmatrix}\end{footnotesize} 
\right)=
\begin{footnotesize}\begin{pmatrix}
0\\
-\frac{4}{15} w_1 w_2 w_3 w_4 w_5
\end{pmatrix}\end{footnotesize}.
\]
Moreover the eigenvalues of the matrix $A$ are $\pm i \sqrt{2}$ with corresponding vectors $\mathbf{p} = \mathbf{q} = \begin{scriptsize}\begin{pmatrix} 1/\sqrt{2} \\ - i/\sqrt{2}\end{pmatrix}\end{scriptsize}$, see \cite[eqn.~(20)]{SoMedeCB07}. Since in our case $G_{21}=0$, cf. \cite[eqn.~(27)]{SoMedeCB07}, we get 
\[
\mathcal{H}_{32} = E \left(\mathbf{q}, \mathbf{q}, \mathbf{q}, \bar{\mathbf{q}}, \bar{\mathbf{q}} \right) = \begin{pmatrix} 0 \\ -\frac{4}{15} \left( - \frac{i}{\sqrt{2}} \right)^3 \left( \frac{i}{\sqrt{2}} \right)^2 \end{pmatrix} = \begin{pmatrix} 0 \\ \frac{i}{15\sqrt{2}}\end{pmatrix}.
\]
From \cite[eqn.~(39)]{SoMedeCB07} we finally obtain 
\[
\ell_2 = \frac{1}{12} \mathrm{Re} \big( \langle \mathbf{p}, \mathcal{H}_{32} \rangle \big) =  - \frac{1}{360},
\]
giving the conclusion. We recall that, for any $\mathbf{v}, \mathbf{w} \in \mathbb{C}^2$, we have $\langle \mathbf{v}, \mathbf{w} \rangle := \sum_{k=1}^2 \bar{v}_k w_k$.

\section{Study of the family of functions $g_{\beta,h}(\lambda)$}\label{App:study_function_g}

We are interested in computing the zeros of the function $g_{\beta,h}(\lambda) = 3 \lambda - \beta [\tanh(\lambda + h) + \tanh(\lambda-h)]$. Equivalently, we look for solutions to the fixed point equation
\begin{equation}\label{Eqn:fixed_point}
\lambda = \Gamma_{\beta, h} (\lambda) \quad \mbox{ with } \quad \Gamma_{\beta, h} (\lambda) = \tfrac{\beta}{3} \left[ \tanh (\lambda+h)+\tanh(\lambda-h) \right] \,.
\end{equation}
It follows from \eqref{Eqn:fixed_point} that
\begin{itemize}
\item $\lambda \longmapsto \Gamma_{\beta,h}(\lambda)$ is a continuous function for all the values of $\beta$ and $h$;
\item $\lim_{\lambda \to \pm \infty} \Gamma_{\beta,h}(\lambda) = \pm \tfrac{2\beta}{3}$;
\item $\Gamma_{\beta,h}'(\lambda) = \frac{\beta}{3} \left[ \frac{1}{\cosh^2(\lambda+h)} + \frac{1}{\cosh^2(\lambda-h)} \right] > 0$ for every $\lambda$, for all values of $\beta$ and $h$.
\end{itemize}
Since $\Gamma_{\beta, h} (\lambda)$ is an odd function with respect to $\lambda$, we have $\Gamma_{\beta, h} (0) = 0$ for all $\beta$ and $h$, so that \eqref{Eqn:fixed_point} always admits the solution $\lambda = 0$. Now, we investigate under what conditions positive solutions $\lambda > 0$ may occur. We restrict to work in the positive half-line.\\
In general, if
\begin{equation}\label{Cond:2nd_threshold}
\Gamma_{\beta, h}' (0) = \frac{2\beta}{3\cosh^2(h)} > 1\,,
\end{equation}
then there is at least one positive  solution. However, since $\Gamma_{\beta, h} (\lambda)$ is not always concave, there may be a positive solution even when \eqref{Cond:2nd_threshold} fails. In this case, there must be at least two positive solutions (corresponding to the curve $\lambda \longmapsto \Gamma_{\beta, h} (\lambda)$, crossing the diagonal first from below and then from above). We study the sign of the second order derivative to have a more precise picture. We compute
\[
\Gamma_{\beta,h}''(\lambda) = -\frac{2 \beta}{3} \left[ \frac{\tanh(\lambda+h)}{\cosh^2(\lambda+h)} + \frac{\tanh(\lambda-h)}{\cosh^2(\lambda-h)} \right] \,.
\]
Notice that $\Gamma_{\beta,h}''(0) = 0$ for every $\beta$ and $h$. Therefore $\lambda = 0$ is an inflection point for all the values of the parameters. We search for other possible inflection points. We have
\[
\Gamma_{\beta,h}''(\lambda) \geq 0 \quad \Longleftrightarrow \quad \sinh(\lambda+h)\cosh^3(\lambda-h)+\sinh(\lambda-h)\cosh^3(\lambda+h) \leq 0. 
\]
%
By using the definitions of hyperbolic sine and cosine, after a few algebraic manipulations, we get
\[
\Gamma_{\beta,h}''(\lambda) \geq 0 \quad \Longleftrightarrow \quad \cosh(2\lambda) \leq \frac{\cosh(4h)-3}{2\cosh(2h)}.
\]
The last inequality is never satisfied if $\frac{\cosh(4h)-3}{2\cosh(2h)} < 1$, while it is equivalent to $\lambda \leq \frac{1}{2} \operatorname{arccosh} \left[ \frac{\cosh(4h)-3}{2\cosh(2h)} \right]$ whenever $\frac{\cosh(4h)-3}{2\cosh(2h)} \geq 1$.
Since, by exploiting the identity $\cosh(4h) = 2\cosh^2(2h)-1$, we obtain 
%
%
\[
\frac{\cosh(4h)-3}{2\cosh(2h)} \geq 1 \quad \Longleftrightarrow \quad \cosh^2(2h) -\cosh(2h) -2 \geq 0 \quad \Longleftrightarrow \quad   h \geq \frac{1}{2} \ln (2 + \sqrt{3}),
\]
it follows that
\begin{itemize}
\item 
for $h < \frac{1}{2} \ln (2 + \sqrt{3})$, the function $\lambda \longmapsto \Gamma_{\beta,h}(\lambda)$ is strictly concave for $\lambda > 0$;
\item 
for $h > \frac{1}{2} \ln (2 + \sqrt{3})$, there is a positive inflection point at $\lambda_{\mathrm{I}} = \frac{1}{2} \operatorname{arccosh} \left[ \frac{\cosh(4h)-3}{2\cosh(2h)} \right]$ such that the function $\lambda \longmapsto \Gamma_{\beta,h}(\lambda)$ is strictly convex for $0 < \lambda < \lambda_{\mathrm{I}}$ and strictly concave for $\lambda > \lambda_{\mathrm{I}}$; 
\item 
for $h = \frac{1}{2} \ln (2 + \sqrt{3})$,  it yields $\lambda_{\mathrm{I}} = 0$.
\end{itemize}

To conclude, the function $\Gamma_{\beta,h}$ has at most one inflection point and therefore it changes curvature at most once. As a consequence, we obtain the following results concerning the number of positive solutions of the fixed point equation \eqref{Eqn:fixed_point}. 

\begin{enumerate}[label={\footnotesize (F\arabic*)},ref={\footnotesize (F\arabic*)}]
\item \label{fact:1}
If $h \leq h_{\mathrm{tc}}$ and $\beta \leq \beta_{\mathrm{c}}(h)$, then the curve $\Gamma_{\beta,h}(\lambda)$ is strictly concave on $(0,+\infty)$ and hence there is no intersection with the diagonal.
\item 
If $h > h_{\mathrm{tc}}$ and $\beta < \beta_{\mathrm{c}}(h)$, the function $\Gamma_{\beta,h}$ changes its curvature either below or above the diagonal, giving rise to none or two positive fixed points. As the mapping $\beta \mapsto \Gamma_{\beta,h}$ is increasing, the two regions are delimited by the separation curve $\beta_{\mathrm{T}}(h)$ ($\leq \beta_{\mathrm{c}}(h)$), corresponding to the choice of parameters for which there exists $\lambda_\ast > 0$ such that $\Gamma_{\beta,h}(\lambda_\ast) = \lambda_\ast$ and $\Gamma_{\beta,h}'(\lambda_\ast) = 1$. More precisely, we have
\begin{enumerate}[label={\footnotesize \roman*.},ref={\footnotesize (F2\roman*)}]
\item \label{fact:2a}
for $\beta < \beta_{\mathrm{T}}(h)$ there is no intersection with the diagonal;
\item \label{fact:2b}
for $\beta_{\mathrm{T}}(h) < \beta < \beta_{\mathrm{c}}(h)$, the diagonal and the curve $\Gamma_{\beta,h}$ intersect two times. 
\end{enumerate}
\item \label{fact:3}
If $h > h_{\mathrm{tc}}$ and $\beta = \beta_{\mathrm{c}}(h)$, there is exactly one positive solution of \eqref{Eqn:fixed_point}. 
\item \label{fact:4}
If $\beta > \beta_{\mathrm{c}}(h)$, no matter the curvature, the function $\Gamma_{\beta,h}(\lambda)$ crosses the diagonal at precisely one positive $\lambda$.
\end{enumerate}

\begin{remark*}
Recall that the number of solutions to the fixed point equation \eqref{Eqn:fixed_point} corresponds to the number of zeroes of the function $g_{\beta, h}$ and that, in turn, this number relates to the number of possible limit cycles for \eqref{MKV:Lienard} via Li{\'e}nard's theorem and generalizations \cite[Sect.~3.8]{Per01}.
Observe that the existence of two periodic orbits is consistent with property \ref{fact:2b}, in the sense that having a second positive zero of $g_{\beta,h}$ breaks the condition for uniqueness of the periodic orbit prescribed by Li{\'e}nard's theorem. However, having two positive zeroes is necessary to possibly have more than one limit cycle, but not sufficient to have exactly two. Roughly speaking, for the latter instance to happen, we need that the well relative to the minimum of $g_{\beta,h}$ reaches a sufficient depth as $\beta$ varies (precise conditions are given in \cite[Thm.~B]{Oda96}). The threshold value $\beta_{\mathrm{T}}$ is the value of $\beta$ for which the function $g_{\beta,h}$ is tangent to the $\lambda$-axis and, therefore, after which there are two positive intersections with such an axis; whereas, $\beta_{\star}$ is the value of $\beta$ at which the well of the minimum becomes sufficiently deep so that each zero corresponds to a periodic solution of \eqref{MKV:Lienard}. As a consequence, the value $\beta_{\mathrm{T}}$ is a \emph{lower bound} for $\beta_{\star}$. See Figure~\ref{Fig:PD+LB}.
\end{remark*}

\begin{figure}[h!]
\centering
\subfigure[]{\includegraphics[scale=0.7]{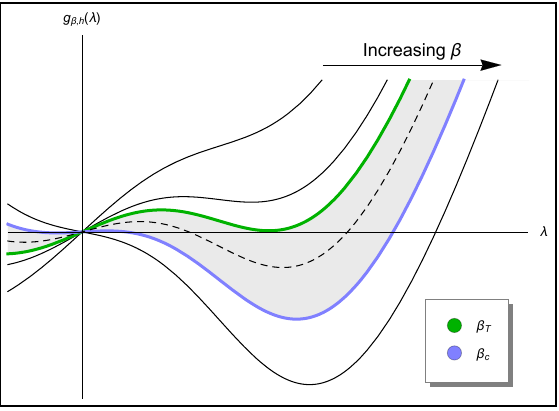}} \:
\subfigure[]{\includegraphics[scale=0.784]{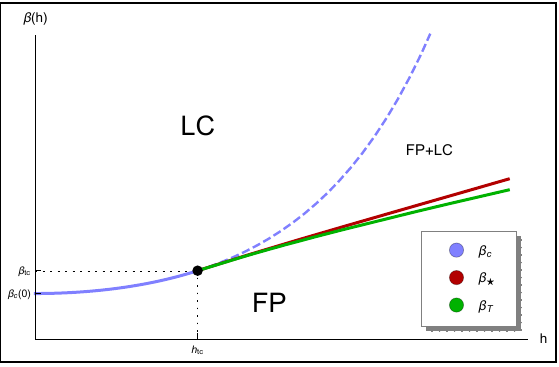}}
\caption{In panel (a) the dependence of the family $g_{\beta,h}$ on the parameter $\beta$ is shown for a fixed value of $h$. The curves corresponding to $\beta_T$ (green solid) and $\beta_c$ (blue solid) are highlighted. They are respectively the threshold values for the transitions $0$-to-$2$ and $2$-to-$1$ crossings between $g_{\beta,h}$ and the $\lambda$-axis. The curve corresponding to $\beta=\beta_{\star}$, where periodic solutions appear, lies somewhere in the shadowed region.
In panel (b), we display the separation curves in the phase diagram of system \eqref{MKV:Lienard} together with the lower bound $\beta_{\mathrm{T}}$ for the threshold value $\beta_{\star}$. Both the curves $\beta_{\mathrm{T}}(h)$ and $\beta_{\star}(h)$, for $h \geq h_{\mathrm{tc}}$, have been obtained numerically.
}
\label{Fig:PD+LB}
\end{figure}

\bigskip



%
\paragraph{Acknowledgments.} The authors are grateful to the anonymous referees whose comments and remarks led to a significant improvement of the paper. The authors wish to thank Paolo Dai Pra for fruitful and inspiring  discussions. FC was supported by The Netherlands Organisation for Scientific Research (NWO) via TOP-1 grant 613.001.552. 
%



\bibliographystyle{abbrv}

\end{document}